\documentclass[a4paper,11pt]{article}
\usepackage{amsmath}
\usepackage{amsfonts}
\usepackage{supertabular}
\usepackage[latin9]{inputenc}
\usepackage[english]{varioref}
\usepackage{dcolumn}
\usepackage[height=22cm , width = 16cm , top = 4cm , left = 3cm, a4paper]{geometry}
\usepackage[a4paper]{geometry}
\usepackage[final]{graphicx}
\usepackage{epsfig}
\usepackage{pstricks}
\usepackage{psfrag}
\usepackage{rotating}
\usepackage{supertabular}
\usepackage{booktabs}
\usepackage{delarray}
\usepackage{rotating}
\usepackage{subfigure}
\usepackage{nextpage}
\usepackage{layout}
\usepackage{amsthm}
\usepackage{dsfont}
\usepackage{hyperref}
\usepackage[hypcap]{caption}
\usepackage{color}
\numberwithin{equation}{section}

\theoremstyle{plain}
\newtheorem{thm}{Theorem}[section]

\newtheorem{lem}[thm]{Lemma}

\newtheorem{defn}[thm]{Definition}

\newcommand{\enter}{\bigskip}

\date{}
\begin{document}
 \author{ Prasanta Kumar Barik and {Ankik Kumar Giri}\footnote{Corresponding author. Tel +91-1332-284818 (O);  Fax: +91-1332-273560  \newline{\it{${}$ \hspace{.3cm} Email address: }}ankikgiri.fma@iitr.ac.in}\\
\footnotesize Department of Mathematics, Indian Institute of Technology Roorkee,\\ \small{ Roorkee-247667, Uttarakhand, India}\\
  }

\title{Weak solutions to the continuous coagulation model with collisional breakage}

\maketitle

\hrule \vskip 11pt

\begin{quote}
{\small {\em\bf Abstract.} A global existence theorem on weak solutions is shown for the continuous coagulation equation with collisional breakage under certain classes of unbounded collision kernels and distribution functions. This model describes the  dynamics of particle growth when binary collisions occur to form either a single particle via coalescence or two/more particles via breakup with possible transfer of mass. Each of these processes may take place with a suitably assigned probability depending on the volume of particles participating in the collision. The distribution function may have a possibility to attain an algebraic singularity for small volumes. \enter
}
\end{quote}
\noindent
{\bf Keywords:} Coalescence; Collisional breakage; Weak compactness; Existence.\\
{\rm \bf MSC (2010).} Primary: 45K05, 45G99, Secondary: 34K30.\\

\vskip 11pt \hrule

\section{Introduction}
Coagulation and breakage processes arise in the different fields of science and engineering, for instance, chemistry (when a matter (water vapor) changes from its gas phase to a liquid phase by condensation process, the molecules in the gas start to come together to form bigger and bigger droplets (dew drops) of the liquid phase), astrophysics (formation of the planets), atmospheric science (raindrop breakup), biology (aggregation of red blood cells), etc. The basic reactions between particles taken into account are the coalescence of a pair of particles to form bigger particles and the breakage of particles into smaller pieces. In general, coagulation event is always a nonlinear process. However, the breakage process may be divided into two different categories on the basis of breakage behaviour of particles, $(i)$ \emph{linear breakage} and $(ii)$ \emph{collisional or nonlinear breakage}. Due to the external forces or spontaneously (that depends on the nature of particles), linear breakage occurs whereas the collisional breakage happens due to the collision between a pair of particles. It is worth to mention that the smaller particles are only produced due to the linear breakage process while the collisional breakage allows some transfer of mass between a pair of particles and might produce particles of mass larger than one of each colliding particles. Here, the volume (or size) of each particle  is denoted by a positive real number. 
 Now, let us turn to the mathematical model considered in this work. We first take a closed system of particles undergoing binary collisions such that any number of particles are produced by the collision, subject to the constraint that the sum of the volumes of the product particles is equal to the sum of the volumes of the two original particles. The following three possible outcomes may arise in such a process;

 \begin{itemize}
   \item if only one particle is produced by the collision, then a coagulation event occurs,
   \item if the collision process gives two particles, then the collision was either elastic or volume (or mass) was exchanged between the original particles,
   \item  if three or more particles emerge from the collision, then a breakage event takes place.
 \end{itemize}
  The continuous coagulation and collisional breakage model has been studied in \cite{Brown:1995, Safronov:1972, Vigil:2006, Wilkins:1982} to describe the evolution of raindrops in clouds.
If the particle size distribution is represented by the number density $g=g(z, t)$ for volume $z \in \mathbb{R}_{+}:=(0, \infty)$ at time $t \in[0, \infty)$, the continuous coagulation equation with collisional breakage read, as
\begin{eqnarray}\label{ccecfe}
\frac{\partial g}{\partial t}  =\mathcal{C}(g)+\mathcal{B}(g),
\end{eqnarray}
  where the coalescence term $\mathcal{C}(g):= \mathcal{C}_1(g)-\mathcal{C}_2(g)$,
\begin{eqnarray*}\label{C1}
\mathcal{C}_1 (g)(z, t) := \frac{1}{2}\int_0^z E(z-z_1, z_1 ) \Phi (z-z_1, z_1)g(z-z_1, t)g(z_1, t)dz_1,
\end{eqnarray*}

\begin{eqnarray*}\label{C2}
\mathcal{C}_2 (g)(z, t) : = \int_{0}^{\infty} E(z, z_1 )  \Phi(z, z_1)g(z, t)g(z_1, t)dz_1,
\end{eqnarray*}
and breakup term $\mathcal{B}(g):=\mathcal{B}_1(g)-\mathcal{B}_2(g)$,
 \begin{align*}
\mathcal{B}_1 (g)(z, t) :=  & \frac{1}{2} \int_{z}^{\infty} \int_{0}^{z_1}P(z|z_1-z_2;z_2)[1- E(z_1-z_2, z_2 )]\nonumber\\
 &~~~~~~~~~~~~~~~\times \Phi(z_1-z_2, z_2)g(z_1-z_2, t)g(z_2, t)dz_2dz_1,
\end{align*}

\begin{eqnarray*}\label{B2}
\mathcal{B}_2 (g)(z, t): = \int_{0}^{\infty} [1- E(z, z_1 )] \Phi(z, z_1)g(z, t)g(z_1, t)dz_1.
\end{eqnarray*}
  Adding $\mathcal{C}_2$ and $\mathcal{B}_2$, we obtain

  \begin{eqnarray*}\label{B2}
\mathcal{B}_3(g)(z, t) : = \int_{0}^{\infty}  \Phi(z, z_1)g(z, t)g(z_1, t)dz_1.
\end{eqnarray*}

Hence, (\ref{ccecfe}) can also be written in the following equivalent form
  \begin{align}\label{ccecfe1}
\frac{\partial g}{\partial t}  =\mathcal{C}_1(g)-\mathcal{B}_3(g)+\mathcal{B}_1(g),
\end{align}
with the following initial data
\begin{align}\label{in1}
g(z,0) = g_{0}(z)\geq 0\ \ \mbox{a.e.}
\end{align}
Here, $\Phi(z, z_1)$ denotes the collision kernel, which describes the rate at which  particles of volumes $z$ and $z_1$ are colliding and $E(z, z_1 )$ is the probability that the two colliding particles aggregate to form a single one. If they do not (an event which occurs with probability $1- E(z, z_1 )$) they undergo breakage with possible transfer of mass. In addition, both the collision kernel $\Phi$ and collision probability $E$ are symmetric in nature, i.e. $\Phi(z, z_1)=\Phi(z_1, z)$ and $E(z, z_1)= E(z_1, z)$ with $0 \leq E(z, z_1) \leq 1$, $ \forall  (z, z_1) \in \mathbb{R}_{+} \times \mathbb{R}_{+}$. Next, $P(z|z_1;z_2)$ is a distribution function describing the expected number of particles of volume $z$ produced from the breakage event arising from the collision of particles of volumes $z_1$ and $z_2$. \\

The first integral $\mathcal{C}_1(g)$ and the second integral $\mathcal{C}_2(g)$ of (\ref{ccecfe}) represent the formation and disappearance, respectively, of particles of volume $z$ due to coagulation events. On the other hand, the third integral $\mathcal{B}_1(g)$ represents the birth of particles of volumes $z$ due to the collisional breakage between a pair of particles of volumes $z_1-z_2$ and $z_2$, 
 and the last integral $\mathcal{B}_2(g)$  describes the death of particles of volume $z$ due to collisional breakage between a pair of particles of volumes $z$ and $z_1$. The factor $1/2$ appears in the integrals $\mathcal{C}_1(g)$ and $\mathcal{B}_1(g)$ to avoid double counting of formation of particles due to coagulation and collisional breakage processes.\\

The distribution function $P$ has the following properties:\\
$(i)$ $P$ is non-negative and symmetric with respect to $z_1$ and $z_2$, i.e. $P(z|z_1; z_2)=P(z|z_2; z_1) \geq 0$,\\

$(ii)$ The total number of particles resulting from the collisional breakage event is given by
\begin{align}\label{Total particles}
\int_{0}^{z_1+z_2}P(z|z_1; z_2)dz = N,\ \text{for all}\  z_1>0\ \text{and}\ z_2>0,\  P(z|z_1;z_2)=0\  \text{for}\  z> z_1+z_2,
\end{align}

$(iii)$ A necessary condition for mass conservation during collisional breakage events is
\begin{align}\label{mass1}
\int_{0}^{z_1+z_2} z P(z|z_1;z_2)dz = z_1+z_2,\ \ \text{for all}\ z_1>0\ \text{and}\ z_2>0.
\end{align}
 From the condition (\ref{mass1}), the total volume $z_1+z_2$ of particles remains conserved during the collisional breakage of particles of volumes $z_1$ and $z_2$.\\

Next, let us mention some particular cases of the continuous coagulation and collisional breakage equation. When $E(z, z_1)=1$, then equation (\ref{ccecfe1}) becomes the continuous \emph{Smoluchowski coagulation equation} \cite{Barik:2017, Barik:2018, Giri:2010A, Giri:2012}.
 Another case taken into consideration is the collision between a pair of particles of volumes $z$ and $z_1$ that results in either the coalescence of both into of volumes $(z+z_1)$ or into an elastic collision leaving the incoming clusters unchanged. In both cases $P(z|z;z_1)=P(z_1|z;z_1)=1$ and $P(z^{\ast}|z;z_1)=0$ if $z^{\ast} \notin \{ z, z_1 \} $ which again, reduces (\ref{ccecfe1}) into the continuous Smoluchowski coagulation equation with $(E(z, z_1) \Phi(z, z_1))$ as the coagulation rate. Now, by substituting $E=0$ and $P(z|z_1;z_2)=\chi_{[z, \infty)}(z_1)B(z|z_1; z_2) + \chi_{[z, \infty)}(z_2)B(z|z_2; z_1)$ into (\ref{ccecfe1}), it can easily be seen that (\ref{ccecfe1}) becomes the \emph{pure nonlinear breakage model} which has been extensively studied in many articles, \cite{Cheng:1990, Cheng:1988, Matthieu:2007, Kostoglou:2000, Laurencot:2001}. In these articles, the authors have been considered when a pair of particles collide, one particle fragment into smaller pieces without transfer of masses from other one. The continuous nonlinear breakage equation reads as
\begin{align}\label{nonlinear breakage}
\frac{\partial g(z, t)}{\partial t}  = &- \int_{0}^{\infty} \Psi (z, z_1)g(z, t)g(z_1, t)dz_1\nonumber\\
 & + \int_{z}^{\infty}\int_{0}^{\infty} B(z|z_1; z_2) \Psi(z_1, z_2) g(z_1, t)g(z_2, t)dz_2 dz_1,
\end{align}
where $\Psi(z, z_1)=\Psi(z_1, z) \geq 0$ is the collisional kernel and  $B(z|z_1; z_2)$ denotes the breakup kernel or breakage function, which represents particle of volume $z$ obtained by collision between particles of $z_1$ and $z_2$ and satisfies the following property
\begin{align*}
\int_0^{z_1}zB(z|z_1; z_2)dz =z_1,\ \ \ z <z_1 \in \mathbb{R}_{+} \ \text{and}\ z_2 \in \mathbb{R}_{+}.
\end{align*}

Finally, we define moments of number density $g$. Let $\mathcal{M}_r(t)$ denotes the $r^{th}$ moment of $g$ which is defined as
\begin{align*}
\mathcal{M}_r(t):=\mathcal{M}_r(g(z, t)) := \int_0^{\infty} z^r g(z, t)dz,\ \ \text{ where}\ \ r \geq 0.
\end{align*}
The zeroth and first moments represent the total number of particles and the total mass of particles, respectively. In collisional breakage events, the total number of particles, i.e. $\mathcal{M}_0(t)$, increases whereas $\mathcal{M}_0(t)$ decreases during coagulation events. In addition, it is expected that the total mass of the system remains constant during these events. However, sometimes the mass conserving property breaks down due to the rapid growth of coagulation kernels, ($E\Phi$), compare to the breakage kernels, ($[1-E]\Phi $). Hence, \emph{gelation} may appear in the system.\\

In this work, we mainly address the issue on the existence of weak solutions to the continuous coagulation and collisional breakage equation (\ref{ccecfe1})--(\ref{in1}). The existence and uniqueness of solutions to the classical coagulation-fragmentation equations have been discussed in several articles by applying various techniques, see \cite{Barik:2017, Barik:2018, Giri:2010A, Stewart:1989, Stewart:1990}. However, best to our knowledge, the mathematical theory on the continuous coagulation and collisional breakage equation has not been rigorously studied. Although there are a few articles available which are devoted to (\ref{ccecfe1})--(\ref{in1}), see \cite{Brown:1995, Laurencot:2001, Safronov:1972, Vigil:2006, Wilkins:1982}. This model has been described in \cite{Safronov:1972, Wilkins:1982}. In particular, in \cite{Laurencot:2001}, the existence of mass conserving weak solutions to the discrete version of (\ref{ccecfe1})--(\ref{in1}) has been shown by using a weak $L^1$ compactness method. Moreover, they have also studied the uniqueness of solutions, long time behaviour in some particular cases and occurrence of gelation transition. In \cite{Brown:1995}, the structural stability of the continuous coagulation and collisional breakage model is studied by applying both analytical method and numerical experiment. Later in \cite{Vigil:2006}, the partial analytical solutions to the discrete (\ref{ccecfe1})--(\ref{in1}) is studied for the constant collision kernel. Moreover, this solution is also compared with Monte-Carlo simulation. In addition, there are a few articles in which analytical solutions to the continuous nonlinear breakage equations have been investigated for some specific collision kernels only, see \cite{Cheng:1990, Cheng:1988, Matthieu:2007, Kostoglou:2000}. However, in general, it is quite delicate to handle the continuous nonlinear breakage equation mathematically because here the small sized particles have the tendency to fragment further into very small sized clusters which leads to the formation of an infinite number of clusters in a finite time. In order to overcome this situation, we consider a fully nonlinear continuous coagulation and collisional breakage model (\ref{ccecfe1}). Best to our knowledge, this is the first attempt to show the existence of global weak solutions to the continuous coagulation and collisional breakage equation (\ref{ccecfe1})--(\ref{in1}) for large classes of unbounded collision kernels and distribution function.\\

The paper is organized in the following manner: In Section 2, we state some definitions, assumptions and lemmas, which are essentially required in subsequent sections. The statement of main existence theorem is also given at the end of this section. Section 3 contains the rigorous proof of the existence theorem which relies on a weak $L^1$ compactness method.

\section{Definitions and Results}
Let us define the following Banach space $\mathcal{S}^+$ as
\begin{align*}
\mathcal{S}^+:=\{  g \in L^1(\mathbb{R}_{+}, dz): \|g\|_{L^1(\mathbb{R}_{+}, (1+z)dz)} < \infty\ \text{and}\ g \geq 0\ \text{a.e.} \},
 \end{align*}
where
\begin{align*}
\|g\|_{L^1(\mathbb{R}_{+}, (1+z)dz)}:=\int_{0}^{\infty}(1+z)|g(z)|dz.
\end{align*}
 We can also define the norms in the following way:
\begin{align*}
\|g\|_{L^1(\mathbb{R}_{+}, zdz)}:=\int_{0}^{\infty}z|g(z)|dz
\end{align*}
and
\begin{align*}
\|g\|_{L^1(\mathbb{R}_{+}, dz)}:=\int_{0}^{\infty}|g(z)|dz,\ \text{where}\ g\in \mathcal{S}^+.
\end{align*}

 Next, we formulate weak solutions to (\ref{ccecfe1})--(\ref{in1}) through the following definition:
\begin{defn}\label{def1} Let $T \in \mathbb{R}_{+}$. A solution $g$ of (\ref{ccecfe1})--(\ref{in1}) is a non-negative continuous function $g: [0,T]\to \mathcal{S}^+$ such that, for a.e. $z\in \mathbb{R}_{+}$ and all $t\in [0,T]$,\\
$(i)$  the following integrals are finite
     \begin{align*}
     &\int_{0}^{t}\int_{0}^{\infty}\Phi(z, z_1)g(z_1,s)dz_1ds<\infty,\ \ \mbox{and}\ \ \int_{0}^{t} \mathcal{B}_1(g)(z, s)ds<\infty, 
     \end{align*}
$(ii)$  the function $g$ satisfies the following weak formulation to (\ref{ccecfe1})--(\ref{in1})
\begin{align*}
g(z,t)=& g_0(z)+\int_{0}^{t} [\mathcal{C}_1(g)(z, s)-\mathcal{B}_3(g)(z, s)+\mathcal{B}_1(g)(z, s)] ds.
\end{align*}
\end{defn}

Now, throughout the paper, we assume the following conditions on collision kernel $\Phi$, distribution function $P$, and the probability function $E$:\\

$(\Gamma_1)$ $\Phi$ is  non-negative measurable function on $\mathbb{R}_{+} \times \mathbb{R}_{+}$,\\
\\
$(\Gamma_2)$  $\Phi(z, z_1) =  k_1 (z^{\alpha}{z_1}^{\beta}+{z_1}^{\alpha}z^{\beta})$ for all $(z, z_1)\in \mathbb{R}_{+} \times \mathbb{R}_{+}$, $0< \alpha \leq \beta <1$ and for some constant $k_1 \geq 0$,\\
\\
$(\Gamma_3)$ $E$ satisfies the following condition locally:
\begin{align*}
\frac {N-2}{N-1} \leq E(z, z_1) \leq 1, \ \ \ \ \forall (z, z_1) \in (0, 1) \times (0, 1),
\end{align*}
where $N$ is given in (\ref{Total particles}),\\
\\
$(\Gamma_4)$  for each $W>0$ and for $z_1 \in (0, W),$ $0 < \alpha \leq \beta < 1$ (introduce in $(\Gamma_2)$) and any measurable subset $U$ of  $(0, 1)$ with $|U|
\leq \delta$, we have
\begin{align*}
\int_{0}^{z_1} \chi_{U}(z)   P(z|z_1-z_2; z_2) dz \leq \Omega_1(|U|) z_1^{-\alpha}, \  \text{where}\  \lim_{\delta \to 0} \Omega_1 ( \delta )=0,
\end{align*}
where $|U|$ denotes the Lebesgue measure of $U$ and $\chi_{U}$ is the characteristic function of $U$ given by
\begin{equation*}
\chi_{U}(z):=\begin{cases}
1,\ \ & \text{if}\ z\in U, \\
0,\ \ &  \text{if}\ z\notin U,
\end{cases}
\end{equation*}
\\
$(\Gamma_5)$ for $z_1+z_2 > W,$ we have $P(z|z_1; z_2) \leq k(W)z^{-\tau_2} $ for $z\in (0, W)$, where $\tau_2 \in [0,1)$ and $k(W)>0$.\\

 Let us take the following example of distribution function $P$ which satisfies $(\Gamma_4)$--$(\Gamma_5)$.
  \begin{align*}
  P(z|z_1;z_2)= (\nu +2)\frac{z^{\nu}}{{(z_1+z_2)}^{\nu +1}},\ \text{where}\ \  -2 <\nu \leq 0\  \text{and}\  z<z_1+z_2.
 \end{align*}
For $\nu =0$, this leads to the case of binary breakage and for $-1< \nu \leq 0$, we get the finite number of particles, which is denoted by $N$ and written as $N = \frac{\nu +2}{\nu +1} $. But, for $-2< \nu < -1$, we obtain an infeasible number of particles and for the case of $\nu =-1$, an infinite number of daughter particles are produced. It is clear from (\ref{Total particles}).\\

 Now, $(\Gamma_4)$ is checked in the following way: for $z_1\in (0, W)$ and $W>0$ is fixed,
\begin{align*}
\int_{0}^{z_1} \chi_{U}(z)  P(z|z_1-z_2; z_2)dz = (\nu +2)  \int_{0}^{z_1} \chi_{U}(z) \frac{z^{\nu}}{{z_1}^{\nu +1}} dz.
\end{align*}
For $\alpha <1$, and applying H\"{o}lder's inequality, we get
\begin{align*}
\int_{0}^{z_1} \chi_{U}(z)  P(z|z_1-z_2; z_2)dz \leq & (\nu +2) {z_1}^{-\nu -1}  |U|^{\alpha} \bigg( \int_0^{z_1} z^{ \frac{\nu}{1-\alpha}}dz \bigg)^{1-\alpha}\nonumber\\
=&(\nu +2) {z_1}^{-\nu -1}  |U|^{\alpha} \bigg(\frac{{z_1}^{\frac{\nu}{1-\alpha} +1}}{\frac{\nu}{1-\alpha} +1} \bigg)^{1-\alpha}, \mbox{~for~} \nu-\alpha > -1 \nonumber\\
\leq & (\nu +2) \bigg(\frac {1-\alpha }{\nu +1-\alpha}\bigg)^{1-\alpha} |U|^{\alpha} z_1^{-\alpha}.
\end{align*}
 This implies that
\begin{align*}
\int_{0}^{z_1} \chi_{U}(z) P(z|z_1-z_2;z_2)dz \leq \Omega_1 (|U| )z_1^{-\alpha}.
\end{align*}

 In order to verify the $(\Gamma_5),$ for $ {z_1+z_2} > W $ and $W>0$ is fixed, we have
\begin{align*}
 P(z|z_1; z_2)= (\nu +2)\frac{z^{\nu}}{{(z_1+z_2)}^{\nu +1}}\leq (\nu +2) \frac{z^{\nu}}{W ^{1+\nu}}\leq  k(W) z^{-\tau_2},
\end{align*}
where  $-1<\nu \leq 0$, $\tau_2 =-\nu \in [0,1)$  and  $k(W)\geq \frac{\nu +2}{W^{1+\nu}}$.\\

Now we are in the position to state the following existence result:
\begin{thm}\label{existmain theorem1}
Suppose that $(\Gamma_1)$--$(\Gamma_5)$ hold and assume that the initial value $g_0\in \mathcal{S}^+$. Then, (\ref{ccecfe1})--(\ref{in1}) has a weak solution $g$ on $[0, \infty)$ in the sense of Definition \ref{def1}. Moreover,\\ $\|g(t)\|_{L^1(\mathbb{R}_{+}, zdz)} \le \|g_0\|_{L^1(\mathbb{R}_{+}, zdz)}$ and $g \in \mathcal{S}^+ $.
\end{thm}

\section{Existence of weak solutions}\label{existexistence1}
 In order to construct weak solutions to (\ref{ccecfe1})--(\ref{in1}), we follow a weak $L^1$ compactness method introduced in the classical work of Stewart \cite{Stewart:1989}.\\

To prove theorem \ref{existmain theorem1}, we first write (\ref{ccecfe1})--(\ref{in1}) into the limit of a sequence of truncated equations obtained by replacing the collision kernel $\Phi$ by their cut-off kernels $\Phi_n$  \cite{Stewart:1989}, where
\begin{eqnarray}\label{cutoff kernel3}
\Phi_n(z,z_1):=\Phi(z, z_1)\chi_{(0, n)}(z+z_1),
\end{eqnarray}

for $n\ge 1$ and $n \in \mathbb{N}$. Here $\chi_{A}$ denotes the characteristic function on a set $A$.
Considering $(\Gamma_1)$--$(\Gamma_5)$ and $g_0\in \mathcal{S}^+$, for each $n\ge 1$, we may employ the argument of the classical fixed point theory, as in \cite[Theorem 3.1]{Stewart:1989} or \cite{Walker:2002}, to show that
\begin{align}\label{trun cccecf}
\frac{\partial g^n}{\partial t}  = \mathcal{C}_1^n(g^n) - \mathcal{B}_3^n(g^n)+ \mathcal{B}_1^n(g^n), 
\end{align}
where
\begin{align*}
 & \mathcal{C}_{1}^n(g^n)(z,t):=\frac{1}{2}\int_{0}^{z} E(z-z_1, z_1 )  \Phi_n(z-z_1, z_1)g^n(z-z_1,t)g^n(z_1,t)dz_1,\\
 &\mathcal{B}_{3}^n(g^n)(z,t):=\int_{0}^{n-z}\Phi_n(z, z_1)g^n(z,t)g^n(z_1,t)dz_1,\\
 & \mathcal{B}_{1}^n(g^n)(z,t):=\frac{1}{2} \int_{z}^{n}\int_{0}^{z_1}  P(z|z_1-z_2;z_2)[1- E(z_1-z_2, z_2 )]\\
   &~~~~~~~~~~~~~~~~~~~~~~~~\times \Phi_n(z_1-z_2, z_2)g^n(z_1-z_2, t)g^n(z_2, t)dz_2dz_1,
\end{align*}
with the truncated initial data
\begin{equation}\label{trunc ccecfin1}
g^{n}_0(z):=g_0(z) \chi_{(0, n)}(z),
\end{equation}
has a unique non-negative solution $\hat{g}^n\in \mathcal{C}'([0, \infty);L^1((0,n), dz))$ s.t. the truncated version of mass conservation holds, i.e.
\begin{align}\label{trunc mass1}
\int_{0}^{n}z\hat{g}^n(z,t)dz=\int_{0}^{n}z\hat{g}^n_0(z)dz, \ \forall t \ge 0.
\end{align}

Now, we extend the truncated solution $\hat{g}^n$ by zero in $\mathbb{R}_{+} \times [0, T]$, as
\begin{equation}\label{trunc soln}
\hat{g}^{n}(z, t):=\begin{cases}
g^n(z, t),\ \ & \text{if}\ 0 < z< n, \\
\text{0},\ \ &  \text{if}\ z\geq n,
\end{cases}
\end{equation}
for $n\ge 1$ and $n \in \mathbb{N}$. For the no loss of generality, we drop $\hat{.}$ for the remaining part of the article.\\

Next, we wish to establish suitable bounds to apply Dunford-Pettis theorem [\cite{Edwards:1965}, Theorem 4.21.2] and then  equicontinuity of the sequence $(g^n)_{n\in \mathbb{N}}$ in time to use the \textit{Arzel\`{a}-Ascoli Theorem} \cite[Appendix A8.5]{Ash:1972}. This is the aim of the coming subsection.

\subsection{ Weak compactness}
\begin{lem}\label{compactness1}
Assume that $(\Gamma_1)$--$(\Gamma_5)$ hold and fix $T>0$. Let $g_0 \in \mathcal{S}^+$ and $g^n$ be a solution to (\ref{trun cccecf})--(\ref{trunc ccecfin1}). Then, the followings hold true: \\
$(i)$ there is a constant $\mathcal{G}(T)>0$ (depending on $T$) such that
\begin{align*}
\int_{0}^{n}(1+z)g^n(z,t)dz\leq \mathcal{G}(T)\ \ \text{for}\ \  n\ge 1\ \ \text{and all} \ \ t\in [0,T],
\end{align*}
$(ii)$ for any given $\epsilon> 0$, there exists $W_\epsilon>0$ (depending on $\epsilon$) such that, for all $t\in[0,T]$\\
\begin{align*}
\sup_{n\ge 1} \left\{ \int_{W_\epsilon}^{\infty}g^n(z,t)dz \right\}\leq \epsilon,
\end{align*}
$(iii)$ for a given $\epsilon > 0$, there exists  $\delta_\epsilon>0$ (depending on $\epsilon$) such that, for every measurable set $U$ of $\mathbb{R}_{+}$ with $|U|\leq \delta_\epsilon$, $n \ge 1$ and $t\in [0,T]$,\\
\begin{align*}
\int_{U}g^n(z, t)dz< \epsilon.
\end{align*}
\end{lem}

\begin{proof}
$(i)$  Let $n\geq 1$ and $t\in [0,T]$, where $T>0$ is fixed. For $n=1$, the proof is trivial. Next, for $n>1$ and then taking integration of (\ref{trun cccecf}) from $0$ to $1$  with respect to $z$ and by using Leibniz's rule, we obtain
\begin{align}\label{Unibound1}
\frac{d}{dt} \int_0^1  g^n(z,t)dz = \int_0^1 \mathcal{C}_1^n(g^n)(z, t) dz - \int_0^1 \mathcal{B}_3^n(g^n)(z, t) dz + \int_0^1 \mathcal{B}_1^n(g^n)(z, t) dz.
\end{align}
The first term on the right-hand side of (\ref{Unibound1}) can be simplified by using Fubini's theorem and the transformation $z-z_1={z}'$ and $z_1=z_1'$ as
\begin{align}\label{Unibound2}
\int_0^1 \mathcal{C}_1^n(g^n)(z, t) dz 
 =&\frac{1}{2} \int_0^1\int_{0}^{1-z_1} E(z, z_1) \Phi_n(z, z_1)g^n(z,t)g^n(z_1,t)dzdz_1.
\end{align}
Using Fubini's theorem, the third term on the right-hand side of (\ref{Unibound1}) can be written as
\begin{align}\label{Unibound3}
\int_0^1 \mathcal{B}_1^n(g^n)(z, t) dz
 =&\frac{1}{2}  \int_0^1 \int_0^{z_1} \int_0^{z_1}  P(z|z_1-z_2; z_2) [1- E(z_1-z_2, z_2)]\nonumber\\
 &~~~~~~~~~\times \Phi_n(z_1-z_2, z_2) g^n(z_1-z_2, t)g^n(z_2, t)dz_2 dz dz_1\nonumber\\
&+\frac{1}{2}  \int_1^n \int_0^{1} \int_0^{z_1}  P(z|z_1-z_2; z_2) [1- E(z_1-z_2, z_2)]\nonumber\\
&~~~~~~~\times \Phi_n(z_1-z_2, z_2) g^n(z_1-z_2, t)g^n(z_2, t)dz_2 dz dz_1=:I_1^n+I_2^n.
\end{align}
Let us manipulate $I_1^n$, by changing the order of integrations, using (\ref{Total particles}) and applying the transformation $z_1-z_2=z_1'$ and $z_2=z_2'$, as
\begin{align*} 
I_1^n 
=&\frac{N}{2}  \int_0^1 \int_0^{z_1} [1- E(z_1-z_2, z_2)]\Phi_n(z_1-z_2, z_2) g^n(z_1-z_2, t)g^n(z_2, t) dz_2 dz_1\nonumber\\
= &\frac{N}{2}  \int_0^1 \int_0^{1-z} [1- E(z, z_1)] \Phi_n(z, z_1) g^n(z, t)g^n(z_1, t) dz_1 dz.
\end{align*}

Next, simplifying $I_2^n$, by using Fubini's theorem and (\ref{Total particles}), as
\begin{align}\label{Unibound33}
I_2^n 
\leq & \frac{1}{2}  \int_1^n  \int_0^{z_1} \int_0^{z_1} P(z|z_1-z_2; z_2) [1- E(z_1-z_2, z_2)]\nonumber\\
&~~~~~~~~~~~~~~~~~\times \Phi_n(z_1-z_2, z_2) g^n(z_1-z_2, t)g^n(z_2, t)dz dz_2 dz_1 \nonumber\\
= & \frac{N}{2}  \int_1^n  \int_0^{z}  [1- E(z-z_1, z_1)] \Phi_n(z-z_1, z_1) g^n(z-z_1, t)g^n(z_1, t) dz_1 dz.
\end{align}
Again using Fubini's theorem and applying transformation $z-z_1=z'$ and $z_1=z_1'$ into (\ref{Unibound33}), we get
\begin{align*}
I_2^n \leq   & \frac{N}{2}  \int_0^1  \int_1^{n}  [1- E(z-z_1, z_1)] \Phi_n(z-z_1, z_1) g^n(z-z_1, t)g^n(z_1, t) dz dz_1\nonumber\\
& + \frac{N}{2}  \int_1^n  \int_{z_1}^{n}  [1- E(z-z_1, z_1)] \Phi_n(z-z_1, z_1) g^n(z-z_1, t)g^n(z_1, t) dz dz_1\nonumber\\
 = & \frac{N}{2}  \int_0^1  \int_{1-z}^{n-z}  [1- E(z, z_1)] \Phi_n(z, z_1) g^n(z, t)g^n(z_1, t) dz_1 dz \nonumber\\
& + \frac{N}{2}  \int_1^n  \int_{0}^{n-z}  [1- E(z, z_1)] \Phi_n(z, z_1) g^n(z, t)g^n(z_1, t) dz_1 dz.
\end{align*}
Substituting the estimates on $I_1^n$ and $I_2^n$ into (\ref{Unibound3}), we evaluate
\begin{align}\label{Unibound35}
\int_0^1 \mathcal{B}_1^n(g^n)(z, t) dz
\leq &\frac{N}{2}  \int_0^1 \int_0^{1-z} [1- E(z, z_1)] \Phi_n(z, z_1) g^n(z, t)g^n(z_1, t) dz_1 dz\nonumber\\
 &+ \frac{N}{2}  \int_0^1  \int_{1-z}^{n-z}  [1- E(z, z_1)] \Phi_n(z, z_1) g^n(z, t)g^n(z_1, t) dz_1 dz \nonumber\\
& + \frac{N}{2}  \int_1^n  \int_{0}^{n} \Phi_n(z, z_1) g^n(z, t)g^n(z_1, t) dz_1 dz.
\end{align}

Inserting (\ref{Unibound2}) and (\ref{Unibound35}) into (\ref{Unibound1}), we obtain  
 \begin{align}\label{Unibound4}
\frac{d}{dt} \int_0^1  g^n(z,t)dz
\leq & - \int_0^1\int_{0}^{1-z} [1-\frac{1}{2}  E(z, z_1) -\frac{N}{2} (1- E(z, z_1)) ] \Phi_n(z, z_1)g^n(z,t)g^n(z_1,t)dz_1dz\nonumber\\
&- \int_0^1\int_{1-z}^{1} [1-\frac{N}{2} (1- E(z, z_1))]  \Phi_n(z, z_1)g^n(z,t)g^n(z_1,t)dz_1dz\nonumber\\
&- \int_0^1\int_{1}^{n-z} \Phi_n(z, z_1)g^n(z,t)g^n(z_1,t)dz_1dz\nonumber\\
&+ \frac{N}{2}  \int_0^1  \int_{1}^{n}   \Phi_n(z, z_1) g^n(z, t)g^n(z_1, t) dz_1 dz \nonumber\\
& + \frac{N}{2}  \int_1^n  \int_{0}^{n}   \Phi_n(z, z_1) g^n(z, t)g^n(z_1, t) dz_1 dz.
\end{align}
Applying $(\Gamma_3)$ to the first and the second integrals and then using the negativity of the first, second and third terms on the right-hand side of \eqref{Unibound4}, we have
\begin{align}\label{Unibound5}
\frac{d}{dt} \int_0^1  g^n(z,t)dz \leq & N  \int_0^1  \int_{1}^{n}  \Phi_n(z, z_1) g^n(z, t)g^n(z_1, t) dz_1 dz \nonumber\\
& + \frac{N}{2}  \int_1^n  \int_{1}^{n} \Phi_n(z, z_1) g^n(z, t)g^n(z_1, t) dz_1 dz.
\end{align}
Applying $(\Gamma_2)$, (\ref{trunc ccecfin1}) and $g_0\in \mathcal{S}^+$ to (\ref{Unibound5}), we obtain
\begin{align}\label{Unibound6}
\frac{d}{dt} \int_0^1  g^n(z,t)dz
\leq & Nk_1  \int_0^1  \int_{1}^{n}  [z^{\alpha} z_1 +z^{\beta} z_1] g^n(z, t)g^n(z_1, t) dz_1 dz \nonumber\\
& + Nk_1  \int_1^n  \int_{1}^{n}  z z_1 g^n(z, t)g^n(z_1, t) dz_1 dz\nonumber\\
\leq & 2Nk_1  \|g_0 \|_{L^1(\mathbb{R}_{+}, zdz)} \int_0^1   g^n(z, t) dz
 + Nk_1 \|g_0 \|_{L^1(\mathbb{R}_{+}, zdz)}^2.
\end{align}
Again, taking integration of (\ref{Unibound6}) from $0$ to $t$ with respect to time and then applying Gronwall's inequality, we have
\begin{align}\label{Unibound7}
\int_0^1 g^n(z,t)dz \leq \mathcal{G}_1(T),
\end{align}
where
\begin{align*}
\mathcal{G}_1(T):=\| g_0 \|_{L^1(\mathbb{R}_{+}, dz)} \bigg[ e^{2Nk_1T \| g_0 \|_{L^1(\mathbb{R}_{+}, zdz)}} -1 \bigg].
\end{align*}
Now, using (\ref{Unibound7}), (\ref{trunc mass1}) and (\ref{trunc ccecfin1}), estimate the following integral as
\begin{align*}
\int_0^n(1+z) g^n(z,t)dz =&\int_0^1 g^n(z,t)dz+\int_1^n g^n(z,t)dz+\int_0^n z g^n(z,t)dz\nonumber\\
\leq & \int_0^1 g^n(z,t)dz+2\| g_0 \|_{L^1(\mathbb{R}_{+}, zdz)} \leq \mathcal{G}(T),
\end{align*}
where $\mathcal{G}(T):=\mathcal{G}_1(T)+2 \| g_0\|_{L^1(\mathbb{R}_{+}, zdz)} $. This completes the proof of the first part of Lemma \ref{compactness1}.\\

$(ii)$ The second part of Lemma \ref{compactness1} can be easily proved in similar way as given in Giri et al. \cite{Giri:2010A}.\\

$(iii)$ Choose $ \epsilon > 0$ and let $U \subset \mathbb{R}_{+}$. Using Lemma \ref{compactness1} $(ii)$, we can choose $W \in(0, n)$ such that for all $n\in \mathbb{N}$  and $ t \in [0,T]$,
\begin{align}\label{comp3 2}
\int_{W}^{\infty}g^n(z,t)dz <  \frac {\epsilon} {2}.
\end{align}
 Fix $W>0$, for $n\ge 1$, $\delta_{\epsilon} \in (0, 1)$ and $t\in [0,T]$, we define
\begin{align*}
 r^n(\delta_{\epsilon},t):=\sup \left\{\int_{0}^{W} \chi_{U}(z)g^n(z,t)dz\ :\ U\subset (0,W) \ \ \text{and} \ \
|U|\leq \delta_{\epsilon} \right\}.
\end{align*}
For $n\ge 1$ and $t\in [0,T]$, it follows from the non-negativity of $g^n$, Fubini's theorem, (\ref{trun cccecf})--(\ref{trunc ccecfin1}) and nonnegativity of the second integral to (\ref{trun cccecf}) that

\begin{align}\label{equiintegrable1}
\int_{0}^{W}  \frac{\partial}{\partial t}  \chi_{U}(z) g^n(z,t)dz
\leq  & \int_{0}^{W} \chi_{U}(z)  \mathcal{C}_1^n(g^n)(z, t) dz+ \int_0^W \chi_{U}(z) \mathcal{B}_1^n(g^n)(z, t) dz\nonumber\\
 =: & J_1^n +J_2^n.
\end{align}
 Then by using $(\Gamma_3)$, Fubini's theorem and applying the transformation $z-z_1=z'$ and $z_1=z_1'$, $J_1^n$ can be estimated, similar to Giri et al. \cite{Giri:2012}, as
\begin{align*}
J_1^n \leq k_1 \mathcal{G}(T)(1+W)r^n (\delta_{\epsilon}, t).
\end{align*} Next, by applying Fubini's theorem twice, $(\Gamma_4)$, $(\Gamma_5)$ and H\"{o}lder's inequality for $p>1$ such that $p \tau_2< 1$, we estimate $J_2^n$ as \begin{align}\label{equiintegrable3}
J_2^n= &\frac{1}{2} \int_{0}^{W} \int_{0}^{z_1}\int_{0}^{z_1}\chi_{U}(z) P(z|z_1-z_2;z_2) [1-E(z_1-z_2, z_2) ]  \Phi_n(z_1-z_2, z_2)\nonumber\\
 &\ \ \ \ \ \ \ \ g^n(z_1-z_2, t)g^n(z_2, t)dz dz_2 dz_1 \nonumber\\
& + \frac{1}{2}\int_{W}^{n}  \int_{0}^{z_1} \int_{0}^{W} \chi_{U}(z) P(z|z_1-z_2;z_2) [1-E(z_1-z_2, z_2)]  \Phi_n(z_1-z_2, z_2)\nonumber\\
&\ \ \ \ \ \ \ \ g^n(z_1-z_2, t)g^n(z_2, t)dz dz_2 dz_1\nonumber\\
\leq  & \frac{1}{2} \Omega_1(|U|) \int_{0}^{W} \int_{0}^{z} z_1^{-\alpha}  \Phi_n(z-z_1, z_1)  g^n(z-z_1, t)g^n(z_1, t)dz_1 dz \nonumber\\
& + \frac{1}{2}   k(W)  {\delta_{\epsilon}}^{\frac{p-1}{p}  \left(  \frac{W^{1-\tau_2p}}   {{1-\tau_2p}} \right)^{\frac{1}{p}}  } \int_{W}^{n} \int_{0}^{z_1} \Phi_n(z_1-z_2, z_2)  g^n(z_1-z_2, t)g^n(z_2, t) dz_2 dz_1.
\end{align}

Again repeated application of Fubini's theorem, $(\Gamma_2)$, Lemma \ref{compactness1} $(i)$, $z-z_1=z'$ and $z_1=z_1'$, we have

\begin{align*}
J_2^n \leq  & \frac{1}{2} k_1 \Omega_1(|U|)  \int_{0}^{W} \int_{0}^{W-z} (z+z_1)^{-\alpha}  (z^{\alpha}z_1^{\beta}+z^{\beta} z_1^{\alpha})  g^n(z, t)g^n(z_1, t) dz_1 dz \nonumber\\
& + \frac{1}{2}  k_1 k(W) {\delta_{\epsilon}}^{\frac{p-1}{p}}  \left(  \frac{W^{1-\tau_2p}}   {{1-\tau_2p}} \right)^{\frac{1}{p}}    \int_{0}^{W} \int_{W-z}^{n-z} (z^{\alpha}z_1^{\beta}+z^{\beta} z_1^{\alpha})  g^n(z, t)g^n(z_1, t)dz_1 dz\nonumber\\
& + \frac{1}{2} k_1  k(W)  {\delta_{\epsilon}}^{\frac{p-1}{p}}  \left(  \frac{W^{1-\tau_2p}}   {{1-\tau_2p}} \right)^{\frac{1}{p}}  \int_{W}^{n} \int_{0}^{n-z} (z^{\alpha}z_1^{\beta}+z^{\beta} z_1^{\alpha})  g^n(z, t)g^n(z_1, t)dz_1 dz\nonumber\\
\leq & \frac{1}{2} k_1 \Omega_1(|U|)  \mathcal{G}(T)  \int_{0}^{W} (1+z^{\beta-\alpha} ) g^n(z, t) dz +2k_1  k(W)  {\delta_{\epsilon}}^{\frac{p-1}{p}}  \left(  \frac{W^{1-\tau_2p}}   {{1-\tau_2p}} \right)^{\frac{1}{p}}      \mathcal{G}(T)^2\nonumber\\
\leq & k_1 \Omega_1(|U|)  \mathcal{G}(T)^2 +2k_1  k(W)  {\delta_{\epsilon}}^{\frac{p-1}{p}}  \left(  \frac{W^{1-\tau_2p}}   {{1-\tau_2p}} \right)^{\frac{1}{p}} \mathcal{G}(T)^2.
\end{align*}


Gathering the above estimates on $J_{1}^n$, $J_{2}^n$ and inserting them into (\ref{equiintegrable1}), and applying Leibniz's rule, we obtain
\begin{align*}
\frac{d}{dt}\int_{0}^{W} \chi_{U}(z)g^n(z,t)dz\leq &k_1  \mathcal{G}(T) (1+W)r^n (\delta_{\epsilon} , t)+ k_1 \Omega_1(|U|)  \mathcal{G}(T)^2\\
& +2k_1  k(W)  {\delta_{\epsilon}}^{\frac{p-1}{p}}  \left(  \frac{W^{1-\tau_2p}}   {{1-\tau_2p}} \right)^{\frac{1}{p}}   \mathcal{G}(T)^2.
\end{align*}

Integrating the above inequality with respect to $t$ and taking supremum over all $U$ such that $U \subset (0,W)$ with $|U|$ $\leq \delta_{\epsilon}$, we estimate
\begin{align*}
r^n(\delta_{\epsilon},t) \leq & r^n(\delta_{\epsilon},0)+k_1  \mathcal{G}(T)(1+W)\int_0^t r^n (\delta_{\epsilon} , s)ds+ k_1 \Omega_1(|U|)  \mathcal{G}(T)^2T\\& +2k_1  k(W) T{\delta_{\epsilon}}^{\frac{p-1}{p}}  \left(  \frac{W^{1-\tau_2p}}   {{1-\tau_2p}} \right)^{\frac{1}{p}}    \mathcal{G}(T)^2,\ \ t\in [0,T].
\end{align*}
An application of Gronwall's inequality finally gives
 \begin{align*}
r^n(\delta_{\epsilon},t) \leq C^{*} (\delta_{\epsilon}, W) \exp(k_1 \mathcal{G}(T)T(1+W)),\ \ t\in [0,T],
\end{align*}
where
 \begin{align*}
 C^{*} (\delta_{\epsilon}, W):= & r^n(\delta_{\epsilon},0)+ k_1 \Omega_1(|U|)   \mathcal{G}(T)^2T+2k_1  k(W) T {\delta_{\epsilon}}^{\frac{p-1}{p}}  \left(  \frac{W^{1-\tau_2p}}   {{1-\tau_2p}} \right)^{\frac{1}{p}} \mathcal{G}(T)^2.
 \end{align*}

This shows that
\begin{align}\label{comp3 3}
\mbox{sup}_{n}\{ r^n(\delta_{\epsilon},t)\}\rightarrow 0~~ \mbox{as}~~ \delta_{\epsilon} \rightarrow 0.
\end{align}
Adding (\ref{comp3 2}) and (\ref{comp3 3}), we thus obtain the required result.
\end{proof}
Hence, from Dunford-Pettis theorem, we have $(g^{n})_{n \in \mathbb{N}}$ is a relatively compact subset of $\mathcal{S}^+$ for each $t \in [0, T]$.

\subsection{Equicontinuity with respect to time in weak sense}

By showing the following lemma, we check the time  equicontinuity in weak sense of the family $\{g^n(t), t\in[0,T]\}$ in $L^1(\mathbb{R}_{+}, dz)$.
\begin{lem}\label{4equicontinuiuty}
Let $\psi \in L^{\infty} (\mathbb{R}_{+})$. Then prove that
\begin{eqnarray*}
\lim_{h \to 0} \sup_{t \in [0, T-h]} \bigg| \int_0^{\infty}  [g^n(z, t+h) - g^n (z, t) ] \psi(z) dz \bigg| =0.
\end{eqnarray*}
\end{lem}
\begin{proof}
Let $\psi \in L^{\infty} (\mathbb{R}_{+})$, $h \in (0, T)$ with $h<1$ and $t\in [0, T-h]$. Choose $1<a<n$ such that $ \frac{2}{a} \mathcal{G}(T) \|\psi\|_{L^{\infty}( \mathbb{R}_{+})}   < h /2$. Next, consider the following integral, by using triangle inequality, as
 \begin{align}\label{4equicontinuiuty1}
\bigg|\int_0^{\infty}  [g^n(z, t+h) - g^n (z, t) ] \psi(z)dz \bigg| \leq & \bigg|\int_0^a  [g^n(z, t+h) - g^n (z, t) ]  \psi(z) dz \bigg|\nonumber\\      &+ \bigg|\int_a^{\infty}  [g^n(z, t+h) - g^n (z, t) ]  \psi(z) dz  \bigg|.
\end{align}

Using \eqref{trun cccecf}, the first integral on the right-hand side to (\ref{4equicontinuiuty1}) can be estimated as
\begin{align}\label{4equicontinuiuty2}
\bigg|\int_0^a    & [g^n(z, t+h) - g^n (z, t) ]  \psi(z) dz       \bigg| \nonumber\\
\leq &  \frac{1}{2} \|\psi\|_{L^{\infty}( \mathbb{R}_{+})}  \int_t^{t+h} \int_0^a \int_0^{z} E(z-z_1, z_1) \Phi_n(z-z_1, z_1) g^n(z-z_1, s)g^n(z_1, s)dz_1 dz ds\nonumber\\
& +  \|\psi\|_{L^{\infty}( \mathbb{R}_{+})}     \int_t^{t+h} \int_0^a \int_0^{n}  \Phi_n(z, z_1) g^n(z, s) g^n(z_1, s) dz_1 dz ds\nonumber\\
& +  \frac{1}{2} \|\psi\|_{L^{\infty}( \mathbb{R}_{+})}   \int_t^{t+h} \int_0^a  \int_z^{n} \int_0^{z_1} P(z|z_1-z_2; z_2) [1-E(z_1-z_2, z_2) ]\nonumber\\
&~~~~~~~~~~~~~~~~~~~~\times  \Phi_n(z_1-z_2, z_2) g^n(z_1-z_2, s) g^n(z_2, s)dz_2 dz_1 dz ds=:\sum_{i=3}^5J_i^n,
\end{align}
where $J_i^n$ represents the first, second and third terms, respectively, on the right-hand side to (\ref{4equicontinuiuty2}), for $i=3,4,5$.\\
Next, $J_3^n$ can be evaluated by applying Fubini's theorem, and using $z-z_1 =z'$ and $z_1 =z_1'$, $(\Gamma_2)$ and Lemma \ref{compactness1} $(i)$, we have
\begin{align*}
J_3^n \leq & k_1 \|\psi\|_{L^{\infty}( \mathbb{R}_{+})}  \int_t^{t+h} \bigg[\int_0^a  (1+z) g^n(z, s) dz \bigg]^2 ds
\leq  k_1 \|\psi\|_{L^{\infty}( \mathbb{R}_{+})} \mathcal{G}(T)^2 h.
\end{align*}
Similarly, $J_4^n$ can be estimated by using $(\Gamma_2)$ and Lemma \ref{compactness1} $(i)$, as
\begin{align*}
J_4^n \leq & 2 k_1 \|\psi\|_{L^{\infty}( \mathbb{R}_{+})}  \int_t^{t+h} \bigg[\int_0^n  (1+z) g^n(z, s) dz \bigg]^2 ds
\leq  2 k_1 \|\psi\|_{L^{\infty}( \mathbb{R}_{+})} \mathcal{G}(T)^2 h.
\end{align*}
Now, $J_5^n$ can be estimated by applying Fubini's theorem twice, and using (\ref{Total particles}) and  $(\Gamma_5)$, we obtain
\begin{align}\label{4equicontinuiuty3}
J_5^n \leq &  \frac{N}{2} \|\psi\|_{L^{\infty}( \mathbb{R}_{+})}   \int_t^{t+h} \int_0^a  \int_0^{z_1}
 \Phi_n(z_1-z_2, z_2) g^n(z_1-z_2, s) g^n(z_2, s)dz_2 dz_1  ds\nonumber\\
  & + \frac{1}{2} k(a) \frac{a^{1-\tau_2}}{1-\tau_2} \|\psi\|_{L^{\infty}( \mathbb{R}_{+})}   \int_t^{t+h}  \int_a^{n} \int_0^{z_1} \Phi_n(z_1-z_2, z_2) g^n(z_1-z_2, s) g^n(z_2, s)dz_2 dz_1  ds
\end{align}
Again, applying Fubini's theorem and using $z_1-z_2 =z_1'$ and $z_2 =z_2'$, $(\Gamma_2)$ and Lemma \ref{compactness1} $(i)$, we estimate (\ref{4equicontinuiuty3}) as
\begin{align*}
J_5^n \leq &  \frac{N}{2} \|\psi\|_{L^{\infty}( \mathbb{R}_{+})}   \int_t^{t+h} \int_0^a  \int_0^{a}
 \Phi_n(z_1, z_2) g^n(z_1, s) g^n(z_2, s)dz_2 dz_1  ds\nonumber\\
  & + \frac{1}{2} k(a) \frac{a^{1-\tau_2}}{1-\tau_2} \|\psi\|_{L^{\infty}( \mathbb{R}_{+})}   \int_t^{t+h}  \int_0^{n} \int_0^{n} \Phi_n(z_1, z_2) g^n(z_1, s) g^n(z_2, s)dz_2 dz_1  ds\nonumber\\
 \leq &  k_1 \|\psi\|_{L^{\infty}( \mathbb{R}_{+})} \mathcal{G}(T)^2h \bigg[N+  k(a)\frac{a^{1-\tau_2}}{1-\tau_2} \bigg].
\end{align*}
Inserting estimates on $J_3^n$, $J_4^n$ and $J_5^n$ into (\ref{4equicontinuiuty2}), we thus obtain
\begin{align}\label{4equicontinuiuty4}
\bigg|\int_0^a     [g^n(z, t+h) - g^n (z, t) ]  \psi(z) dz       \bigg|
\leq  k_1 \|\psi\|_{L^{\infty}( \mathbb{R}_{+})} \mathcal{G}(T)^2h \bigg[3+ N+  k(a)\frac{a^{1-\tau_2}}{1-\tau_2} \bigg].
\end{align}
Next, using Lemma \ref{compactness1} $(i)$, the last term on the right-hand side to (\ref{4equicontinuiuty1}) can be estimated as
\begin{align}\label{4equicontinuiuty5}
\bigg|\int_a^{\infty}    [g^n(z, t+h) - g^n (z, t) ]  \psi(z) dz       \bigg|
\leq & \|\psi\|_{L^{\infty}( \mathbb{R}_{+})}  \frac{1}{a} \int_a^{\infty}   z [g^n(z, t+h) + g^n (z, t)]  dz\nonumber\\
\leq & \|\psi\|_{L^{\infty}( \mathbb{R}_{+})}  \frac{2}{a} \mathcal{G}(T) <h/2.
\end{align}

Using (\ref{4equicontinuiuty4}) and (\ref{4equicontinuiuty5}) into (\ref{4equicontinuiuty1}), we have
\begin{align*}
\bigg|\int_0^{\infty} [g^n(z, t+h) - g^n (z, t)] \psi(z) dz \bigg| \leq & k_1 \|\psi\|_{L^{\infty}( \mathbb{R}_{+})} \mathcal{G}(T)^2h \bigg[3+ N+  k(a)\frac{a^{1-\tau_2}}{1-\tau_2} \bigg]  + \frac{h}{2},
\end{align*}
where $h$ is arbitrary. This completes the proof of Lemma \ref{4equicontinuiuty}.
\end{proof}

Then according to a refined version of the \textit{Arzel\`{a}-Ascoli theorem}, see \cite[Theorem 2.1]{Stewart:1989} or \textit{Arzel\`{a}-Ascoli theorem}, see Ash [\cite{Ash:1972}, page 228], we conclude that there exists a subsequence (${g^{n_k}}$) such that
\begin{align*}
\lim_{n_k\to\infty} \sup_{t\in [0,T]}  \left| \int_0^\infty  [ g^{n_k}(z,t) - g(z,t)]\ \psi(z)\ dz \right|  = 0, \label{vittel}
\end{align*}
for all $T>0$, $\psi \in L^\infty(\mathbb{R}_{+})$ and some $g \in \mathcal{C}_w([0,T]; L^1(\mathbb{R}_{+}, dz))$, where  $\mathcal{C}_w([0, T]; L^1 (\mathbb{R}_{+}, dz))$ is the space of all weakly continuous functions from $[0, T]$ to $L^1 (\mathbb{R}_{+},  dz)$. This implies that
\begin{equation}\label{equicontinuity f}
  g^{n_k}(t) \rightharpoonup g(t)\ \text{in}\  L^1(\mathbb{R}_{+}, dz) \ \text{as}\ n \to \infty,
\end{equation}
converges uniformly for $t \in [0,T]$ to some $g \in \mathcal{C}_w([0,T]; L^1(\mathbb{R}_{+}, dz))$.\\

Next, for any $ m>0$, $t\in [0,T], $ since we have $g^{n_k} \rightharpoonup g$, we obtain
\begin{align*}
\int_{0}^{m}zg(z,t)dz = \lim_{ n_k\rightarrow {\infty}} \int_{0}^{m} z g^{n_k}{(z,t)}dz \leq \|g_0\|_{L^1(\mathbb{R}_{+}, zdz)} < \infty.
\end{align*}
Using (\ref{trunc mass1}), the non-negativity of each $g^{n_k}$ and $g$, then as $ m\rightarrow \infty $ implies that $\|g( t)\|_{L^1(\mathbb{R}_{+}, zdz)} \le \|g_0\|_{L^1(\mathbb{R}_{+}, zdz)}$ and $g(t) \in \mathcal{S}^+$.

\subsection{Convergence of approximated integrals}\label{subs:limit}
Now, we prove that the limit function $g$ obtained in (\ref{equicontinuity f}) is indeed a weak solution to (\ref{ccecfe1})--(\ref{in1}).
We then have the following result:
\begin{lem}\label{convergence lemma1}
Let $(g^{n})_{n\in \mathbb{N}}$ be a bounded sequence in $\mathcal{S}^+$ and $g\in \mathcal{S}^+$, where $\|g^n\|_{L^1( \mathbb{R}_{+}, (1+z)dz)} \leq \mathcal{G}(T)$ and $g^n\rightharpoonup g$ in $L^1(\mathbb{R}_{+}, dz)$ as $n\to \infty $. Then, for each $a> 0$, we have
\begin{equation}\label{luchon}
 \mathcal{C}_{1}^n(g^n) \rightharpoonup  \mathcal{C}_{1}(g),\ \mathcal{B}_{3}^n(g^n)  \rightharpoonup \mathcal{ B}_{3}(g)\  \text{and}\ \ \mathcal{B}_{1}^n(g^n)  \rightharpoonup  \mathcal{B}_{1}(g)\   \text{in} \  L^1((0, a), dz)\  \text{as}\  n\to \infty.
\end{equation}
\end{lem}

\begin{proof} Let $0<a<n$, where $ z\in (0, a]$. Suppose $\psi \in L^\infty(0, a)$.
It can easily be shown with slide modifications, as in \cite{Giri:2010, Giri:2012, Stewart:1989}, that $\mathcal{C}_{1}^n(g^n) \rightharpoonup  \mathcal{C}_{1}(g)$ and $\mathcal{B}_{3}^n(g^n)  \rightharpoonup  \mathcal{B}_{3}(g)$.\\

  In order to show that $\mathcal{B}_{1}^n(g^n)  \rightharpoonup  \mathcal{B}_{1}(g)\   \text{in} \  L^1((0, a), dz)\  \text{as}\  n\to \infty$, it is sufficient to prove the following integral tends to zero as $n \to \infty$.
\begin{align}\label{passintegrals}
\bigg| \int_{0}^{a}  \psi(z) &[ \mathcal{B}_1(g^n)(z, t)-\mathcal{B}_1(g)(z,t)] dz\bigg|\nonumber\\
 \le & \frac{1}{2} \bigg|\int_{0}^{a} \int_{z}^{b}\int_{0}^{z_1} \psi(z) P(z|z_1-z_2;z_2)[1-E(z_1-z_2,z_2)] \Phi(z_1-z_2,z_2)\nonumber\\
 & ~~~~~~~~\times [g^n(z_1-z_2,t)g^n(z_2,t)-g(z_1-z_2,t)g(z_2,t)]dz_2dz_1 dz \bigg| \nonumber\\
 &+ \frac{1}{2} \bigg|\int_{0}^{a} \int_{b}^{\infty}\int_{0}^{z_1} \psi(z) P(z|z_1-z_2;z_2)[1-E(z_1-z_2,z_2)] \Phi(z_1-z_2,z_2)\nonumber\\
&~~~~~~~~\times [g^n(z_1-z_2,t)g^n(z_2,t)-g(z_1-z_2,t)g(z_2,t)]dz_2dz_1 dz\bigg|,
 \end{align}

 where we choose $b$ with $n>b> a$ large enough for a given $\epsilon >0$, such that
\begin{align}\label{conv21}
\frac{2k_2 k(a) a^{1-\tau_2}} {1-\tau_2} \| \psi \|_{L^{\infty}(0, a)}[  \mathcal{G}(T)^2+\|g\|^2_{L^1(\mathbb{R}_{+}, (1+z)dz)} ](1+b)^{\beta -1}< \frac{\epsilon}{2}.
\end{align}
Now, let us simplify the following integral by applying Fubini's theorem, as
\begin{align*}
\frac{1}{2} \bigg| \int_0^a &\int_z^b \int_0^{z_1}    \psi(z)P(z|z_1-z_2; z_2) [1-E(z_1-z_2, z_2)] \Phi(z_1-z_2, z_2)\nonumber\\
 &    \times [g^n(z_1-z_2, t)g^n(z_2, t)-g(z_1-z_2, t)g(z_2, t)]dz_2dz_1dz \bigg|  \nonumber\\
 \le & \bigg| \frac{1}{2}   \int_0^a \int_0^{z_1} \int_0^{z_1}   \psi(z) P(z|z_1-z_2; z_2) [1-E(z_1-z_2, z_2)] \Phi(z_1-z_2, z_2) \nonumber\\
 &    \times   g^n(z_1-z_2, t)   [g^n(z_2, t)-g(z_2, t)] dz_2dzdz_1 \bigg| \nonumber\\
 & + \bigg| \frac{1}{2} \bigg|   \int_a^b \int_0^a \int_0^{z_1}   \psi(z) P(z|z_1-z_2; z_2) [1-E(z_1-z_2, z_2)] \Phi(z_1-z_2, z_2) \nonumber\\
 &    \times   g^n(z_1-z_2, t)   [g^n(z_2, t)-g(z_2, t)] dz_2dzdz_1 \bigg| \nonumber
 \end{align*}
 \begin{align}\label{conv27}
  &+  \bigg| \frac{1}{2}    \int_0^a \int_0^{z_1} \int_0^{z_1}  \psi(z) P(z|z_1-z_2; z_2) [1-E(z_1-z_2, z_2)] \Phi(z_1-z_2, z_2)\nonumber\\
 &    \times  g(z_2, t)[g^n(z_1-z_2, t)-g(z_1-z_2, t)] dz_2dzdz_1    \bigg| \nonumber\\
  &+  \bigg| \frac{1}{2}   \int_a^b \int_0^a \int_0^{z_1}   \psi(z) P(z|z_1-z_2; z_2) [1-E(z_1-z_2, z_2)] \Phi(z_1-z_2, z_2)\nonumber\\
 &    \times  g(z_2, t)[g^n(z_1-z_2, t)-g(z_1-z_2, t)] dz_2dzdz_1    \bigg|  =: \sum_{i=1}^4   H_i^n,
\end{align}
where $H_i^n$, for $i=1,2,3,4$, are preceding integrals on the right-hand side to \eqref{conv27}. Each $H_i^n$, for $i=1,2,3,4$ is evaluated individually. Let us first estimate $H_1^n$, by using the repeated applications of Fubini's theorem and the transformation $z_1-z_2=z_1'$ and $z_2=z_2'$, as
\begin{align}\label{H1}
H_1^n 
  = &  \frac{1}{2} \bigg| \int_0^a \int_{0}^{a-z_2} \int_0^{z_1+z_2}   \psi(z) P(z|z_1; z_2) \Phi(z_1, z_2) [1-E(z_1, z_2)]
  g^n(z_1, t)\nonumber\\
  &~~~~~~~\times [g^n(z_2, t)-g(z_2, t)] dz dz_1 dz_2 \bigg|.
\end{align}

 Next, the following integral can be estimated, by using Lemma \ref{compactness1} $(i)$, \eqref{Total particles} and $(\Gamma_2)$, for each $z_2 \in (0, a)$, as
\begin{align}\label{H111}
 \frac{1}{2}  \int_{0}^{a-z_2} \int_0^{z_1+z_2} &  \psi(z) P(z|z_1; z_2) \Phi(z_1, z_2) [1-E(z_1, z_2)]
  g^n(z_1, t)  dz dz_1 \nonumber\\
 \leq &  N k_1 \|\psi \|_{L^{\infty}(0, a)}  \int_{0}^{a-z_1} (1+z_1) (1+z_2) g^n(z_1, t) dz_1 \nonumber\\
  \leq &  N k_1 \|\psi \|_{L^{\infty}(0, a)} \mathcal{G}(T) (1+z_2) \in  L^{\infty}(0, a).
\end{align}

From (\ref{H111}) and weak convergence of $g^n \rightharpoonup g$ in $L^1(\mathbb{R}_+, dz)$, we obtain
\begin{align}\label{H1111}
H_1^n \to 0,\ \ \ \text{as}\ \ n \to \infty.
\end{align}

Further, $H_2^n$ can be simplified by using Fubini's theorem, and  the transformation $z_1-z_2=z_1'$ and $z_2=z_2'$, as
\begin{align}\label{H2}
H_2^n 
  = &  \frac{1}{2} \bigg| \int_0^a  \int_{a-z_2}^{b-z_2} \int_0^{a}   \psi(z) P(z|z_1; z_2) [1-E(z_1, z_2)] \Phi(z_1, z_2)  g^n(z_1, t)   [g^n(z_2, t)-g(z_2, t)]dzdz_1  dz_2  \nonumber\\
   & + \frac{1}{2} \int_a^b  \int_0^{b-z_2} \int_0^{a}   \psi(z) P(z|z_1; z_2) [1-E(z_1, z_2)] \Phi(z_1, z_2) \nonumber\\
 &    \times   g^n(z_1, t)   [g^n(z_2, t)-g(z_2, t)] dzdz_1  dz_2 \bigg|.
\end{align}
Next, we evaluate the following integral by using  $(\Gamma_5)$,  $(\Gamma_2)$ and Lemma \ref{compactness1} $(i)$, as
\begin{align}\label{H21}
 \frac{1}{2} \int_{a-z_2}^{b-z_2}  \int_0^{a} &  \psi(z) P(z|z_1; z_2) [1-E(z_1, z_2)] \Phi(z_1, z_2)g^n(z_1, t)dzdz_1  \nonumber\\
 \leq &  \|\psi\|_{L^{\infty} (0, a)} k(a) \frac{ a^{1-\tau_2}}{1-\tau_2} k_1  \int_0^b   (1+z_1)(1+z_2) g^n(z_1, t)  dz_1  \nonumber\\
 \leq &  \|\psi\|_{L^{\infty} (0, a)} k(a) \frac{ a^{1-\tau_2}}{1-\tau_2} k_1 (1+z_2) \mathcal{G}(T) \in L^{\infty}(0, a)\ \text{for each }\ z_2\in (0, a).
\end{align}
Similarly, by using  $(\Gamma_5)$,  $(\Gamma_2)$ and Lemma \ref{compactness1} $(i)$, we estimate the following term, as
\begin{align}\label{H22}
 \frac{1}{2}  \int_0^{b-z_2} \int_0^{a} &  \psi(z) P(z|z_1; z_2) [1-E(z_1, z_2)] \Phi(z_1, z_2)  g^n(z_1, t)  dzdz_1  \nonumber\\
 \leq &  \|\psi\|_{L^{\infty} (0, a)} k(a) \frac{ a^{1-\tau_2}}{1-\tau_2} k_1  \int_0^a   (1+z_1)(1+z_2) g^n(z_1, t)  dz_1  \nonumber\\
 \leq &  \|\psi\|_{L^{\infty} (0, a)} k(a) \frac{ a^{1-\tau_2}}{1-\tau_2} k_1 (1+z_2) \mathcal{G}(T) \in L^{\infty}(0, b)\ \text{for each }\ z_2\in (a, b).
\end{align}
From (\ref{H21}), (\ref{H22}), (\ref{H2}) and the weak convergence $g^n$ to $g$ in $L^1(\mathbb{R}_+, dz)$, we obtain
\begin{align}\label{H23}
H_2^n \to 0,\ \ \ \text{as}\ \ n \to \infty.
\end{align}

Let us now consider $H_3^n$, after implementing Fubini's theorem twice,  $z_1-z_2=z_1'$ $ \& $ $z_2=z_2'$ and interchanging $z_1 ~\& ~ z_2$, as
\begin{align}\label{H3}
H_3^n 
= &  \frac{1}{2} \bigg| \int_0^a \int_0^{a-z_1} \int_0^{z_1+z_2}  \psi(z) P(z|z_1; z_2) [1-E(z_1, z_2)]\nonumber\\
 &~~~~~~~~\times \Phi(z_1, z_2)  g(z_2, t)[g^n(z_1, t)-g(z_1, t)] dz dz_2 dz_1 \bigg|.
\end{align}

Again, using (\ref{Total particles}) and $(\Gamma_2)$, we evaluate the following integral, as
\begin{align}\label{H33}
\frac{1}{2} \int_0^{a-z_1} \int_0^{z_1+z_2} &  \psi(z) P(z|z_1; z_2) [1-E(z_1, z_2)] \Phi(z_1, z_2)  g(z_2, t) dz dz_2\nonumber\\
\leq &  k_1 \|\psi\|_{L^{\infty}(0, a)} N \int_0^a (1+z_1)(1+z_2) g(z_2, t) dz_2\nonumber\\
    \leq &  k_1 \|\psi\|_{L^{\infty}(0, a)}  N (1+z_1) \|g\|_{L^1(\mathbb{R}_+, (1+z)dz)}  \in L^{\infty}(0, a).
\end{align}

Estimate (\ref{H33}) and the weak convergence of $g^n$ to $g$ in $L^1(\mathbb{R}_+, dz)$ confirms that
\begin{align}\label{H333}
H_3^n \to 0,\ \ \ \text{as}\ \ n \to \infty.
\end{align}

Finally, simplifying $H_4^n$ by using repeated applications of Fubini's theorem, and  $z_1-z_2=z_1'$ $\& $ $z_2=z_2'$, as
\begin{align}\label{H4}
H_4^n = &  \frac{1}{2} \bigg| \int_a^b  \int_0^{z_1} \int_0^{a}   \psi(z) P(z|z_1-z_2; z_2) [1-E(z_1-z_2, z_2)] \Phi(z_1-z_2, z_2) \nonumber\\
 &    \times   g(z_2, t)[g^n(z_1-z_2, t)-g(z_1-z_2, t)] dz dz_2 dz_1  \bigg|    \nonumber\\
  \le &  \frac{1}{2} \bigg| \int_0^b  \int_{0}^{b-z_1} \int_0^{a}   \psi(z) P(z|z_1; z_2) [1-E(z_1, z_2)] \Phi(z_1, z_2) \nonumber\\
 & ~~~~~~   \times   g(z_2, t)[g^n(z_1, t)-g(z_1, t)] dz dz_2 dz_1 \nonumber\\
   & + \frac{1}{2} \int_0^a  \int_0^{a-z_1} \int_0^{a}   \psi(z) P(z|z_1; z_2) [1-E(z_1, z_2)] \Phi(z_1, z_2) \nonumber\\
 &  ~~~~~~  \times   g(z_2, t)[g^n(z_1, t)-g(z_1, t)] dz dz_2 dz_1 \bigg|.
\end{align}

 Similar to $H_2^n$, by using $(\Gamma_5)$,  $(\Gamma_2)$, Lemma \ref{compactness1} $(i)$, \eqref{H4} and  weak convergence $g^n$ to $g$ in $L^1(\mathbb{R}_+, dz)$, it can be easily shown that
\begin{align}\label{H43}
H_4^n \to 0,\ \ \ \text{as}\ \ n \to \infty.
\end{align}

Inserting \eqref{H1111}, \eqref{H23}, \eqref{H333} and \eqref{H43} into (\ref{conv27}), we get
\begin{align}\label{conv28}
\frac{1}{2} \bigg| \int_0^a &\int_z^b \int_0^{z_1}    \psi(z)P(z|z_1-z_2; z_2) [1-E(z_1-z_2, z_2)] \Phi(z_1-z_2, z_2)\nonumber\\
 &    \times [g^n(z_1-z_2, t)g^n(z_2, t)-g(z_1-z_2, t)g(z_2, t)]dz_2dz_1dz \bigg| \to 0\ \ n \to \infty.
\end{align}

 Now, we consider the last term on the right-hand side of \eqref{passintegrals}. By using repeated applications of Fubini's theorem and the transformation $z_1-z_2=z_1'$ and $z_2=z_2'$, we have
\begin{align}\label{conv220}
\frac{1}{2} \bigg| \int_{0}^{a} & \int_{b}^{\infty}\int_{0}^{z_1} \psi(z)P(z|z_1-z_2; z_2) [1-E(z_1-z_2, z_2)] \Phi(z_1-z_2, z_2)\nonumber\\
 & ~~~~~~~~\times [g^n(z_1-z_2, t)g^n(z_2, t)-g(z_1-z_2, t)g(z_2, t)]dz_2dz_1dz \bigg| \nonumber\\
 = &\frac{1}{2} \bigg| \int_{0}^{a}  \int_{0}^{b}\int_{b}^{\infty} \psi(z)P(z|z_1-z_2; z_2) [1-E(z_1-z_2, z_2)] \Phi(z_1-z_2, z_2)\nonumber\\
 & ~~~~~~~~\times [g^n(z_1-z_2, t)g^n(z_2, t)-g(z_1-z_2, t)g(z_2, t)]dz_1 dz_2 dz \nonumber\\
 &+\frac{1}{2}  \int_{0}^{a}  \int_{b}^{\infty}\int_{z_2}^{\infty} \psi(z)P(z|z_1-z_2; z_2) [1-E(z_1-z_2, z_2)] \Phi(z_1-z_2, z_2)\nonumber\\
 & ~~~~~~~~\times [g^n(z_1-z_2, t)g^n(z_2, t)-g(z_1-z_2, t)g(z_2, t)]dz_1 dz_2 dz \bigg| \nonumber\\
 \le & \frac{1}{2} \bigg|  \int_{0}^{b}\int_{b-z_2}^{n} \int_{0}^{a} +         \psi(z)P(z|z_1; z_2) [1-E(z_1,z_2)] \Phi(z_1, z_2)\nonumber\\
 & ~~~~~~~~\times [g^n(z_1, t)g^n(z_2, t)-g(z_1, t)g(z_2, t)]dz  dz_1 dz_2 \bigg| \nonumber\\
 & + \frac{1}{2} \bigg|   \int_{b}^{\infty}\int_{0}^{n} \int_{0}^{a}         \psi(z)P(z|z_1; z_2) [1-E(z_1,z_2)] \Phi(z_1, z_2)\nonumber\\
 & ~~~~~~~~\times [g^n(z_1, t)g^n(z_2, t)-g(z_1, t)g(z_2, t)]dz  dz_1 dz_2 \bigg| \nonumber\\
 &+ \frac{1}{2} \bigg| \int_{0}^{\infty}\int_{n}^{\infty} \int_{0}^{a}          \psi(z)P(z|z_1; z_2) [1-E(z_1,z_2)] \Phi(z_1, z_2)\nonumber\\
 & ~~~~~~~~\times [g^n(z_1, t)g^n(z_2, t)-g(z_1, t)g(z_2, t)]dz  dz_1 dz_2 \bigg|.
\end{align}
 Let us first simplify the first integral on the right-hand side of \eqref{conv220} by applying Fubini's theorem, as
\begin{align}\label{conv2201}
\frac{1}{2} \bigg| \int_{0}^{b} & \int_{b-z_2}^{n} \int_{0}^{a}  \psi(z)P(z|z_1; z_2) [1-E(z_1, z_2)] \Phi(z_1, z_2)\nonumber\\
 & ~~\times [g^n(z_1, t)g^n(z_2, t)-g(z_1, t)g(z_2, t)]dz dz_1 dz_2 \bigg| \nonumber\\
 \le & \frac{1}{2} \bigg| \bigg\{  \int_{0}^{b}\int_{b-z_1}^{b} \int_{0}^{a} +\int_{b}^{n}\int_{0}^{b} \int_{0}^{a} \bigg\} \psi(z)P(z|z_1; z_2) [1-E(z_1, z_2)] \Phi(z_1, z_2)\nonumber\\
 & ~~\times  g(z_2, t)[g^n(z_1, t)-g(z_1, t)]  dz dz_2 dz_1  \bigg| \nonumber\\
 & + \frac{1}{2} \bigg| \bigg\{  \int_{0}^{b}\int_{b-z_2}^{b} \int_{0}^{a} +\int_{0}^{b}\int_{b}^{n} \int_{0}^{a} \bigg\} \psi(z)P(z|z_1; z_2) [1-E(z_1, z_2)] \Phi(z_1, z_2)\nonumber\\
 & ~~\times g^n(z_1, t)[g^n(z_2, t)-g(z_2, t)] dz dz_1 dz_2  \bigg|
\end{align}
Similar to previous cases,  by using $(\Gamma_5)$,  $(\Gamma_2)$, Lemma \ref{compactness1} $(i)$ and  weak convergence $g^n$ to $g$ in $L^1(\mathbb{R}_{+})$, it can be easily shown that
\begin{align}\label{conv2202}
\frac{1}{2} \bigg| \int_{0}^{b} & \int_{b-z_2}^{n} \int_{0}^{a}  \psi(z)P(z|z_1; z_2) [1-E(z_1, z_2)] \Phi(z_1, z_2)\nonumber\\
 & ~~\times [g^n(z_1, t)g^n(z_2, t)-g(z_1, t)g(z_2, t)]dz dz_1 dz_2 \bigg| \to 0\ \ n\to \infty.
 \end{align}
Next, the second integral on the right-hand side of \eqref{conv220} is estimated, by using $(\Gamma_2)$, $(\Gamma_5)$, (\ref{conv21}) and Lemma \ref{compactness1} $(i)$, as
\begin{align}\label{conv22}
\frac{1}{2} \bigg|  \int_{b}^{\infty} & \int_{0}^{n} \int_{0}^{a}  \psi(z)P(z|z_1; z_2) [1-E(z_1,z_2)] \Phi(z_1, z_2)\nonumber\\
 & ~~~~~~~~\times [g^n(z_1, t)g^n(z_2, t)-g(z_1, t)g(z_2, t)]dz  dz_1 dz_2 \bigg| \nonumber\\
\leq &   \frac{1}{2} k(a) \| \psi \|_{L^{\infty}(0, a)}   \int_{b}^{\infty}  \int_{0}^{n} \int_{0}^{a}   z^{-\tau_2} \Phi(z_1, z_2)[g^n(z_1, t)g^n(z_2, t)+ g(z_1, t)g(z_2, t)]dz dz_1dz_2  \nonumber\\
\leq & k_1 k(a) \| \psi \|_{L^{\infty}(0, a)} \int_{b}^{\infty}\int_{0}^{a} z^{-\tau_2} \frac{(1+z_2)}{(1+z_2)^{1-\beta }} \bigg[g^n(z_2,t) \mathcal{G}(T)+g(z_2,t)\|g\|_{L^1(\mathbb{R}_{+}, (1+z)dz)} \bigg] dz dz_2\nonumber\\
\leq & \frac{k_1 k(a) a^{1-\tau_2}} {1-\tau_2} \| \psi \|_{L^{\infty}(0, a)} \bigg[ \mathcal{G}(T)^2+\|g\|^2_{L^1(\mathbb{R}_{+}, (1+z)dz)} \bigg](1+b)^{\beta - 1} < \frac{\epsilon}{2}.
\end{align}

Finally, by applying $(\Gamma_2)$, $(\Gamma_5)$, and $g^n=0$ for $z>n$, the last integral on the right-hand of \eqref{conv220} can be evaluated as
\begin{align*}
\frac{1}{2} \bigg|  \int_{0}^{\infty} & \int_{n}^{\infty} \int_{0}^{a}  \psi(z)P(z|z_1; z_2) [1-E(z_1,z_2)] \Phi(z_1, z_2)\nonumber\\
 & ~~~~~~~~\times [g^n(z_1, t)g^n(z_2, t)-g(z_1, t)g(z_2, t)]dz  dz_1 dz_2 \bigg| \nonumber\\
\leq &   k_1 k(a) \| \psi \|_{L^{\infty}(0, a)}   \int_{0}^{\infty}  \int_{n}^{\infty} \int_{0}^{a}   z^{-\tau_2} (1+z_1)(1+ z_2)g(z_1, t)g(z_2, t)dz dz_1dz_2  \nonumber\\
\leq & \frac{k_1 k(a) a^{1-\tau_2}} {1-\tau_2} \| \psi \|_{L^{\infty}(0, a)}  \|g\|^2_{L^1(\mathbb{R}_{+}, (1+z)dz)} < \infty.
\end{align*}
The above term is bounded uniformly independent with $n$. Then, by Lebesgue dominated convergence theorem, the above integral goes to $0$ as $n \to \infty $, i.e.
\begin{align}\label{conv29}
\frac{1}{2} \bigg|  \int_{0}^{\infty} & \int_{n}^{\infty} \int_{0}^{a}  \psi(z)P(z|z_1; z_2) [1-E(z_1,z_2)] \Phi(z_1, z_2)\nonumber\\
 & ~~~~~~~~\times [g^n(z_1, t)g^n(z_2, t)-g(z_1, t)g(z_2, t)]dz  dz_1 dz_2 \bigg| \to 0.
\end{align}

Inserting estimates \eqref{conv2202}, \eqref{conv22} and \eqref{conv29} into \eqref{conv220}, we obtain
\begin{align}\label{conv30}
\frac{1}{2} \bigg| \int_{0}^{a} & \int_{b}^{\infty}\int_{0}^{z_1} \psi(z)P(z|z_1-z_2; z_2) [1-E(z_1-z_2, z_2)] \Phi(z_1-z_2, z_2)\nonumber\\
 & ~~~~~~~~\times [g^n(z_1-z_2, t)g^n(z_2, t)-g(z_1-z_2, t)g(z_2, t)]dz_2dz_1dz \bigg|  \to 0\ \ \text{as}\ \ n \to \infty.
\end{align}

 Using (\ref{conv28}) and (\ref{conv30}) into (\ref{passintegrals}), we get
\begin{align}\label{pass 1}
\bigg| \int_{0}^{a} & \psi(z) [\mathcal{B}_1(g^n)(z, t)-\mathcal{B}_1(g)(z,t)] dz\bigg| \to 0\ \ n \to \infty.
 \end{align}

We conclude that (\ref{luchon}) holds. Thus, this completes the proof of Lemma \ref{convergence lemma1}.
\end{proof}

Now, we are in a position to prove Theorem \ref{existmain theorem1} by employing above results.

\begin{proof}
 \textit{of Theorem \ref{existmain theorem1}}: Fix $a \in (0, n_k)$, $T>0$  and let us consider $(g^{n_k})_{n\in N}$ be a weakly convergent subsequence of the approximating solutions obtained from (\ref{equicontinuity f}). Hence, from  (\ref{equicontinuity f}) and for  $t\in [0,T]$, we get
\begin{align}\label{exist lim}
 g^{n_k}(z,t) \rightharpoonup  g(z,t)\ \  \text{in}\ \ {L^1((0, a), dz)}\ \ \text{as}\ \ n_k \to \infty.
\end{align}

 Let $ \psi \in \L^{\infty}(0, a)$. Then from Lemma \ref{convergence lemma1}, we have for each $ s \in [0,t]$
\begin{align}\label{main convergence1}
\int_{0}^{a}\psi(z)[(\mathcal{C}_1^{n_k}-\mathcal{B}_3^{n_k}+\mathcal{B}_1^{n_k})(g^{n_k})(z,s)-(\mathcal{C}_1-\mathcal{B}_3+\mathcal{B}_1)(g)(z,s)]dz \to 0 \ \ \text{as} \ \ n_k\to \infty.
\end{align}

In order to apply the dominated convergence theorem, the boundedness of the following integral is shown, by using $(\Gamma_2)$, $(\Gamma_5)$ (\ref{Total particles}), Lemma \ref{compactness1} $(i)$ and the repeated application of Fubini's theorem along with the transformation $z-z_1=z'$, $z_1=z_1'$, as
\begin{align}\label{exist domin}
\bigg|\int_{0}^{a} &\psi(z)[(\mathcal{C}_1^{n_k}-\mathcal{B}_3^{n_k}+\mathcal{B}_1^{n_k})(g^{n_k})(z,s)-(\mathcal{C}_1-\mathcal{B}_3+\mathcal{B}_1)(g)(z,s)] dz\bigg|\nonumber\\
  \leq &  \|\psi\|_{L^{\infty}(0, a)}\int_{0}^{a}\biggl[ \frac{1}{2}\int_0^z \Phi(z-z_1, z_1)[g^{n_k}(z-z_1,s)g^{n_k}(z_1,s)+g(z-z_1,s)g(z_1,s)]dz_1\nonumber\\
 &+\int_0^{n_k-z} \Phi(z, z_1) [g^{n_k}(z,s)g^{n_k}(z_1,s)+g(z,s)g(z_1,s)]dz_1+ \int_{n_k-z}^{\infty} \Phi(z, z_1)g(z,s)g(z_1,s)dz_1\nonumber\\
 &+ \frac{1}{2} \int_{z}^{\infty}\int_{0}^{z_1} P(z|z_1-z_2; z_2) \Phi(z_1-z_2, z_2)\nonumber\\
  &~~~~~~~~\times [g^{n_k}(z_1-z_2, s)g^{n_k}(z_2, s)+g(z_1-z_2, s)g(z_2, s)]dz_2dz_1 \bigg] dz\nonumber \\
\leq &  \|\psi\|_{L^{\infty}(0, a)}\biggl[ \frac{1}{2}\int_{0}^{a} \int_0^a \Phi(z, z_1)[g^{n_k}(z,s)g^{n_k}(z_1,s)+g(z,s)g(z_1,s)]dz_1dz\nonumber\\
 &+\int_{0}^{a}\int_0^{n_k} \Phi(z, z_1) [g^{n_k}(z,s)g^{n_k}(z_1,s)+g(z,s)g(z_1,s)]dz_1 dz \nonumber\\
 &+ \int_{0}^{a}\int_{n_k-z}^{\infty} \Phi(z, z_1)g(z,s)g(z_1,s)dz_1dz\nonumber\\
 &+ \frac{1}{2} \int_{0}^{\infty} \int_{0}^{\infty}\int_{0}^{z_1} P(z|z_1; z_2) \Phi(z_1, z_2) [g^{n_k}(z_1, s)g^{n_k}(z_2, s)+g(z_1, s)g(z_2, s)]dzdz_1 dz_2\bigg] \nonumber \\
 \leq & k_1 \|\psi\|_{L^{\infty}(0, a)}  \bigg[ \mathcal{G}(T)^2(3+N)+ \|g\|^2_{L^1(\mathbb{R}_{+}, (1+z)dz)}(5+N)  \bigg]< \infty.
\end{align}
Since the left-hand side of (\ref{exist domin}) is in $L^{1}((0, a), dz)$, then from (\ref{main convergence1}), (\ref{exist domin}) and the Lebesgue dominated convergence theorem, we obtain
\begin{align}\label{converge1}
\int_0^t\int_{0}^{a} \psi(z)[(\mathcal{C}_1^{n_k}-\mathcal{B}_3^{n_k}+\mathcal{B}_1^{n_k})(g^{n_k})(z,s)-(\mathcal{C}_1-\mathcal{B}_3+\mathcal{B}_1)(g)(z,s)]dzds \to 0 \ \ \text{as} \ \ k\to \infty.
\end{align}
Since $\psi$ is arbitrary and (\ref{converge1}) holds for $ \psi \in L^{\infty}(0, a)$ as $k \to \infty$, hence, by applying Fubini's theorem, we get
\begin{align}\label{exist last1}
\int_{0}^{t}(\mathcal{C}_1^{n_k}-\mathcal{B}_3^{n_k}+\mathcal{B}_1^{n_k})(g^{n_k})(z, s)ds\rightharpoonup  \int_{0}^{t} (\mathcal{C}_1-\mathcal{B}_3+\mathcal{B}_1)(g)(z, s)ds\  \text{ in}\ L^{1}((0, a), dz),\ \text{ as}\ k  \to \infty,
\end{align}
Then, by the definition of $(\mathcal{C}_1^{n_k}-\mathcal{B}_3^{n_k}+\mathcal{B}_1^{n_k})$,  we obtain
\begin{align}\label{exist last2}
g^{n_k}(z,t)= \int_{0}^{t} (\mathcal{C}_1^{n_k}-\mathcal{B}_3^{n_k}+\mathcal{B}_1^{n_k})(g^{n_k})(z, s)ds+ g^{n_k}_0(z),\ \ \ \text{for}\ t \in [0,T]
\end{align}
and thus, it follows from (\ref{exist last1}), (\ref{exist lim}) and (\ref{exist last2}) that
\begin{align*}
\int_{0}^{a}\psi(z)g(z,t)dz = \int_{0}^{t} \int_{0}^{a} \psi(z)(\mathcal{C}_1-\mathcal{B}_3+\mathcal{B}_1)(g)(z, s)dzds+\int_{0}^{a}\psi(z)g_0(z)dz,
\end{align*}
for any $\psi \in L^{\infty}(0, a)$.  Hence for the arbitrariness of $T$ $\& $ $a$, the uniqueness of limit and for all $\psi \in L^{\infty}(0, a)$, we have $g$ is a solution to (\ref{ccecfe1})--(\ref{in1}). This implies that for almost any $z \in (0, a)$, we have
\begin{align*}
g(z,t)=  g_0(z)+\int_{0}^{t} (\mathcal{C}_1-\mathcal{B}_3+\mathcal{B}_1)(g)(z, s)ds.
\end{align*}
We conclude that $g$ is a solution to \eqref{ccecfe1}--\eqref{in1} on $[0, \infty)$. This completes the proof of the existence Theorem \ref{existmain theorem1}.
\end{proof}

\subsection*{Acknowledgment}
PKB would like to thank University Grant Commission (UGC), $6405/11/44$, India, for assisting Ph.D fellowship and AKG wish to thank Science and Engineering Research Board (SERB), Department of Science and Technology (DST), India for their funding support through the project $YSS/2015/001306$ for completing this work.

\end{document}